\documentclass{amsart}
\usepackage{xcolor}
\usepackage{amssymb}
\usepackage{amsmath}
\usepackage{mathtools}

\usepackage[colorlinks=true, linkcolor=red, citecolor=green, urlcolor=blue, pagebackref=true, breaklinks=true]{hyperref}
\usepackage[capitalise]{cleveref}

\newtheorem{theorem}{Theorem}[section]
\newtheorem{lemma}[theorem]{Lemma}
\newtheorem{corollary}[theorem]{Corollary}
\newtheorem{proposition}[theorem]{Proposition}

\theoremstyle{definition}
\newtheorem{definition}[theorem]{Definition}

\newtheorem{conjecture}[theorem]{Conjecture}

\newtheorem{notation}[theorem]{Notation}

\DeclareMathOperator{\maxgen}{maxgen}
\DeclareMathOperator{\mg}{mg}

\begin{document}
\title{On Gotzmann thresholds and a conjecture of Bonanzinga and Eliahou}

\author{Trung Chau}


\subjclass[2020]{}

\begin{abstract}
   We obtain a recursive formula for the Gotzmann threshold of a power of a variable. Consequently, we give an affirmative answer to a conjecture of Bonanzinga and Eliahou.
\end{abstract}

\maketitle

\section{Introduction}

Throughout this article, let $\Bbbk$ be a field and $R_n=\Bbbk[x_1,\dots, x_n]$ the polynomial ring in $n$ variables over $\Bbbk$.

Hilbert functions of homogeneous ideals of $R_n$ play a vital role in commutative algebra and algebraic geometry, and have been an active area of research. In 1927, Macaulay proved a lower bound for the growth of Hilbert functions of homogeneous ideals, and provided special ideals that attain this bound, called \emph{lex ideals}. Lex ideals have been used in many applications in combinatorics, e.g., Sperner theory. We refer to \cite{BL05,En97} for more details. Homogeneous ideals with the slowest growth of Hilbert functions are called \emph{Gotzmann}, a class that include lex ideals, as discussed.

Only a few classes of Gotzmann ideals are known in the literature. An article of Murai and
 Hibi \cite{MH08} is among the rare instances where Gotzmann homogeneous ideals are studied, and shown to have very specific structures. On the other hand, there are more successes in determining Gotzmann monomial ideals, but even in this restricted situation it is widely considered to be a difficult task.  Hoefel and Mermin \cite{HM12} provided an explicit formula for all Gotzmann squarefree monomial ideals.  Murai \cite{Mu07} completely classified all Gotzmann monomial ideals in $R_3$. Our focus will be on Gotzmann principal Borel ideals. A monomial ideal $J\subseteq R_n$ is called \emph{Borel-stable} or \emph{strongly stable} if for any monomial $m\in J$ and any variable $x_j\mid m$, we have $x_i(m/x_j)\in J$ for any $1\leq i\leq j$. Given a monomial $u\in R_n$, let $\langle u \rangle$ denote the smallest monomial ideal in $R_n$ that contains $u$, and let $B(u)$ denote the set of minimal generators of $\langle u \rangle$. Assume that $x_n\nmid u$. Bonanzinga \cite{Bo03} showed that there exists an integer $\tau_n(u) \geq 0$ such that
 \[
 \text{$\langle ux_n^d \rangle$ is Gotzmann $\Leftrightarrow$ $d\geq \tau_n(u)$}.
 \]
 The number $\tau_n(u)$ is called the \emph{Gotzmann threshold} of $u$. Using Murai's result \cite{Mu07}, the Gotzmann threshold of a monomial $x_1^ax_2^b\in R_3$ can be determined:
 \[
 \tau_3(x_1^ax_2^b) = \binom{b}{2}.
 \]
 The next case $n=4$ is completely solved by Bonanzinga and Eliahou \cite{BE21}:
 \[
 \tau_4(x_1^ax_2^bx_3^c) = \binom{\binom{b}{2}}{2} + \frac{b+4}{3}\binom{b}{2} + (b+1)\binom{c+1}{2}+\binom{c+1}{3}-c.
 \]
 The difficulty in determining a formula for the Gotzmann threshold of any monomial seems to skyrocket, as noted by Bonanzinga and Eliahou. A slightly easier problem is to determine the summand with the largest growth. As may be seen from the known formulas, the power of $x_2$ seems to be one  dominant term. For this reason, Bonanzinga and Eliahou \cite{BE24} studied the Gotzmann threshold of only a power of $x_2$, namely $x_2^d$. They obtained
 \[
 \tau_5(x_2^d) = \binom{\binom{\binom{d}{2}}{2}+\binom{d+1}{3}+\binom{d}{2}}{2}
 - \binom{\binom{d}{2}}{3} +\binom{d+3}{4}-d. \]
 With the formulae for $\tau_i(x_2^d)$ avaialble for $i\in \{3,4,5\}$, Bonanzinga and Eliahou proposed the following.

 \begin{conjecture}[\protect{\cite[Conjecture 5.3]{BE24}}]\label{conjec}
     For $n\geq 3$, the Gotzmann threshold $\tau_n(x_2^d)$ is a polynomial of degree $2^{n-2}$ in $d$ with the same dominant term as the $(n-2)$-iterated binomial coefficient
     \[
 \begin{pmatrix}
     \binom{\binom{d}{2}}{2}\\
     \cdots\\
     2
 \end{pmatrix},
 \]
    which is $2(d/2)^{2^{n-2}}$.
 \end{conjecture}

 In this article, we obtain an iterative formula for $\tau_n(x_r^d)$ for any integers $n,r,d$.

 \begin{theorem}\label{thm:main-4-recursive}
    Given integers $n,r$ where $2\leq r<n-1$. Then the Gotzmann threshold $\tau_n(x_r^d)$ in $\Bbbk[x_1,\dots ,x_{n}]$ can be computed recursively:
    \begin{align*}
        \tau_{r+1}(x_r^d) &= \binom{d+r-1}{r}-d,\\
        \tau_{n}(x_r^d)&= \binom{d+n-1}{d} - \binom{d+n-r}{d} - \binom{d+r-1}{d} +1
            + \sum_{i=r+1}^{n-1} (-1)^{n+1-i}\binom{\tau_{i}(x_r^d)}{n+1-i}.
    \end{align*}
\end{theorem}

The case $r=2$ is of interest and presented in the following.

\begin{corollary}\label{cor:r=2}
    Given an integer $n>3$. Then the Gotzmann threshold $\tau_n(x_2^d)$ in $\Bbbk[x_1,\dots ,x_{n}]$ can be computed recursively:
    \begin{align*}
        \tau_{3}(x_2^d) &= \binom{d+1}{2}-d,\\
        \tau_{n}(x_2^d)&= \binom{d+n-2}{n-1} -d+ \sum_{i=3}^{n-1}(-1)^{n+1-i} \binom{\tau_{i}(x_2^d)}{n+1-i}.
    \end{align*}
\end{corollary}

By induction, the answer to Conjecture~\ref{conjec} follows immediately.

\begin{corollary}
    The Gotzmann threshold $\tau_n(x_2^d)$ is a polynomial in $d$ of degree $2^{n-2}$, with the dominant term being the same as that of the $(n-2)$-iterated binomial coefficient
 \[
 \begin{pmatrix}
     \binom{\binom{d}{2}}{2}\\
     \cdots\\
     2
 \end{pmatrix}.
 \]
    In other words, Conjecture~\ref{conjec} holds.
\end{corollary}

We can also recover known formulas for $\tau_n(x_2^d)$ for $n\in [3,5]$, in a different form.

\begin{corollary}
    \begin{align*}
        \tau_3(x_2^d)&=\binom{d+1}{2}-d,\\
        \tau_4(x_2^d)&=\binom{d+2}{3}-d+\binom{\tau_3(x_2^d)}{2},\\
        \tau_5(x_2^d)&=\binom{d+3}{4}-d+\binom{\tau_4(x_2^d)}{2}-\binom{\tau_3(x_2^d)}{3}.
    \end{align*}
\end{corollary}

The proof for our main result, Theorem~\ref{thm:main-4-recursive}, is based on the machinery by Bonanzinga and Eliahou \cite{BE24}.  Mermin has informed the author that Costantini, Francisco, Schweig, and himself have confirmed Conjecture~\ref{conjec} independently \cite{Gotzmann-Mermin}. Their tools and techniques, on the other hand, are different and rely on other~theories.

The article is structured as follows. In Section~\ref{sec:pre}, we recall the concepts, tools, and the machinery to compute Gotzmann thresholds by Bonanzinga and Eliahou. We apply these in Section~\ref{sec:proof} to prove Theorem~\ref{thm:main-4-recursive}.

\section*{Acknowledgements}

The author is supported by the Infosys Foundation. The author would like to thank Manoj Kummini for helpful and productive discussions. The author would like to thank Jeffrey Mermin for encouraging him to complete this project.

\section{Preliminaries}\label{sec:pre}

We recall the machinery developed in \cite{BE24}.

\subsection{Computation of Gotzmann thresholds}

Given a polynomial ring $R_n=\Bbbk[x_1,\dots, x_n]$. Let $S_n$ denote the set of monomials in $R_n$, and $S_{n,d}$ the set of monomials of degree $d$ in $R_n$.

\begin{notation}
    Given a monomial $u\in S_n\setminus \{1\}$ and a set $B$ of monomials of degree at least 1. Set
    \begin{align*}
        \max(u)&= \max\{k\in [n]\colon x_k\mid u  \},\\
        \lambda(u)&= x_{\max(u)},\\
        \maxgen(B)&=\prod_{u\in B} \lambda(u).
    \end{align*}
\end{notation}

We impose the lexicographial order on $S_{n,d}$ induced by $x_1>x_2>\cdots> x_n$.

\begin{definition}[\protect{\cite[Notation 2.5 and 2.6]{BE24}}]
    For $u,v\in S_{n,d}$ where $u\leq v$, set
    \begin{align*}
        L_n(u)&= \{ z\in S_{n,d} \colon z\geq u \},\\
        L^*_n(v,u)&= \{ z\in S_{n,d} \colon v>z\geq u \}=L_n(u)\setminus L_n(v),\\
        \mu_n(u,v)&= \maxgen(L^*_n(v,u)).
    \end{align*}
    $L_n(u)$ and $L^*_n(v,u)$ are called \emph{lexsegments} and \emph{lexintervals}, respectively. The monomial $\mu_n(u,v)$ is called the \emph{cost} of the upward path from $u$ to $v$ in $S_{n,d}$.  We will use the notation
    \[
    u\xrightarrow{\mu_n(u,v)} v
    \]
    for the upward path from $u$ to $v$, along with its cost.
\end{definition}

\begin{notation}[\protect{\cite[Notation 2.9]{BE24}}]
    Given $n\geq 3$ and $u\in S_n$. Set 
    \[
    \mg_n(u) = \maxgen (L_n(u)\setminus B(u)).
    \]
    Note that $\mg_n$ is a function from $S_n$ to itself.  When $n$ is clear, we will simply write~$\mg$.
\end{notation}

\begin{notation}[\protect{\cite[Notation 2.13]{BE24}}]
    Let $n\geq i\geq 1$ be positive integers. Set
    \begin{align*}
        \pi_{n,i}\colon S_n&\to S_i\\
        x_1^{a_1}\cdots x_n^{a_n}&\mapsto x_1^{a_1}\cdots x_i^{a_i}
    \end{align*}
    to be the \emph{truncation map}. When $n$ is clear, we will simply write $\pi_i$.
\end{notation}

\begin{notation}[\protect{\cite[Notation~3.1]{BE24}}]
    Let $\sigma_n$ be defined as follows:
    \begin{align*}
        \sigma_n \colon S_n&\to S_n\\
        \prod_{i=1}^n x_i^{a_i} &\mapsto \prod_{i=1}^n x_i^{a_1+a_2+\cdots+ a_i} = x_1^{a_1}x_2^{a_1+a_2}\cdots x_n^{a_1+a_2+\cdots +a_n}.
    \end{align*}
\end{notation}

The map $\sigma_n$ allows for some easy computation.

\begin{lemma}[\protect{\cite[Proposition 3.4 and Corollary 3.6]{BE24}}]\label{lem:mg=sigma()}
    For any $u= \prod_{i=1}^n x_i^{a_i}\in S_n$ and $t\geq 0$, we have
    \[
    \mg_n(ux_n^t)=\sigma_n^t(\mg_n(u))=\prod_{i=1}^n x_i^{\sum_{j=1}^i a_j\binom{t-1+i-j}{t-1}}.
    \]
\end{lemma}

For the rest of this subsection, assume that $u_0\in S_{n-1}$. We will recall results in \cite{BE24} that allow a computational route to $\tau_n(u_0)$. Spoiler alerts: we need three polynomials $f(t),h(t),k(t)$ associated to $u_0$.

\begin{notation}[\protect{\cite[Notation~3.10]{BE24}}]
    Let $u_0\in S_{n-1}$. For $t\in \mathbb{N}$, set
    \[
    f(t)\coloneqq \deg_{x_n}(\mg_n(u_0x_n^t)).
    \]
\end{notation}

An explicit formula for $f(t)$ is as follows.

\begin{lemma}[\protect{\cite[Theorem~3.11]{BE24}}]\label{lem:f-formula}
    Let $u=x_{i_1}\cdots x_{i_d}\in S_{n-1}$ with $i_1\leq \cdots \leq i_d\leq n-1$ and $n\geq 3$. For $t\geq 0$, we have
    \[
    f(t)=\sum_{k=1}^{r-1}\binom{t+d-k-2+n-i_{k+1}}{n-1-i_{k+1}}\left( |B(x_{i_1}\cdots x_{i_k})| -1  \right).
    \]
\end{lemma}

\begin{definition}[\protect{\cite[Notation 4.10 and Definition 4.11]{BE24}}]
    Let $u_0\in S_{n-1}$. For each $t\geq 0$, let us denote
    \[
    w(t)\coloneqq \pi_{n-1}(\mg_n(u_0x_n^t)),
    \]
    i.e., $\mg_n(u_0x_n^t)=w(t)x_n^{f(t)}$.
    
    By definition, for each $t\geq \tau_{n-1}(u_0)$, there exists a smallest monomial $z_n(t)\geq u_0x_n^t$ such that 
    \[
    \pi_{n-1}\left(\mu_n(u_0x_n^t,z_n(t)) \right) = w(t).
    \]
    For each $t\geq \tau_{n-1}(u_0)$, we then set
    \begin{align*}
        h(t)&= \deg_{x_n} \mu_n(u_0x_n^t,z_n(t)),\\
        k(t)&= \deg_{x_n} z_n(t).
    \end{align*}
\end{definition}

We are finally ready to recall our main tool.

\begin{theorem}[\protect{\cite[Theorem~4.15]{BE24}}]
    Let $u_0$ be a monomial in $R_n$, where $n\geq 3$, such that $x_n\nmid u_0$. Then
    \[
    \tau_n(u_0) = f(0) - h(0)-k(0).
    \]
\end{theorem}

\subsection{Combinatorial formulas}

We record some combinatorial formulae that will be used throughout the article.

\begin{lemma}\label{lem:comb-1}
    Let $a,b,c\geq 0$ be integers. Then
    \[
    \sum_{i=0}^c \binom{a+i}{a} \binom{b+c-i}{b} = \binom{a+b+c+1}{c}.
    \]
\end{lemma}
\begin{proof}
    The result follows from the generating function
    \[
    \sum_{i=0}^\infty \binom{n+i}{n}y^i = \frac{1}{(1-x)^{n+1}}. \qedhere
    \]
\end{proof}

\begin{lemma}\label{lem:comb-2}
    Let $a,b\geq 0$ be integers. Then
    \[
    \sum_{i=0}^b \binom{a+i}{i} = \binom{a+b+1}{b}.
    \]
\end{lemma}

\begin{proof}
    We proceed by induction on $b$. If $b=0$, then both sides equal $1$, as desired. Now assume that the equality holds for $b-1$ for some $b>0$. We then have
    \[
    \sum_{i=0}^b \binom{a+i}{i} = \left(\sum_{i=0}^{b-1} \binom{a+i}{i} \right)+ \binom{a+b}{b} = \binom{a+b}{b-1}+ \binom{a+b}{b} = \binom{a+b+1}{b},
    \]
    where the second equality is due to our induction hypothesis. The result then~follows.
\end{proof}

\section{Proof of Theorem~\ref{thm:main-4-recursive}}\label{sec:proof}

Given two integers $r,n$ where $2\leq r<n$. Fix a monomial $u_0=x_r^d$ in $R_n$. The goal is to find the values of the three polynomials $f(t),h(t),k(t)$ associated to $u_0$, specialized at $t=0$. By definition, the only $x_j$'s that divide $\mg(u_0)$ are in $\{x_i\}_{i\in [r+1,n]}$. Thus we can set $\mg(u_0) = x_{r+1}^{f_{r+1}}x_{r+2}^{f_{r+2}}\cdots x_n^{f_n}$.

\begin{lemma}
    For any $j\in [r+1,n]$, we have
    \[
    f_j =\binom{j+d-2}{j-1} - \binom{j-r+d-1}{d-1}.
    \]
\end{lemma}
\begin{proof}
    It is clear that $|B(x_r^k)|= \binom{k+r-1}{r-1}$ for any $k$. By \cref{lem:f-formula}. we then have
    \begin{align*}
        f_j&= \sum_{i=1}^{d-1}\binom{d-i-2-r+j}{j-r-1} \left( \binom{i+r-1}{r-1}-1 \right)\\
        &= \sum_{i=1}^{d-1}\binom{d-i-2-r+j}{j-r-1}  \binom{i+r-1}{r-1} -\sum_{i=1}^{d-1}\binom{d-i-2-r+j}{j-r-1}\\
        &=\sum_{i=1}^{d-1}\binom{(j-r-1)+(d-1-i)}{j-r-1}  \binom{(r-1)+i}{r-1} -\sum_{i=1}^{d-1}\binom{(j-r-1)+(d-1-i)}{j-r-1}.
    \end{align*}
    Applying \cref{lem:comb-1} to the first sum and \cref{lem:comb-2} to the second sum, we~obtain
    \begin{align*}
        &\ \ \ \  \left( \binom{j+d-2}{j-1} - \binom{j-r+d-2}{j-r-1} \right) -  \left( \binom{j-r+d-2}{d-2}\right)\\
        &=\binom{j+d-2}{j-1} -\left( \binom{j-r+d-2}{d-1}+ \binom{j-r+d-2}{d-2}\right)\\
        &=\binom{j+d-2}{j-1} - \binom{j-r+d-1}{d-1},
    \end{align*}
    as desired.
\end{proof}

By \cref{lem:mg=sigma()}, we obtain $\mg_n(u_0x_n^t)=x_{r+1}^{f_{r+1}(t)}x_{r+2}^{f_{r+2}(t)}\cdots x_n^{f_n(t)}$~where
\begin{align*}
    f_{r+1}(t)&=f_{r+1},\\
    f_{r+2}(t)&=f_{r+2}+f_{r+1}\binom{t}{1},\\
    \cdots\\
    f_{n}(t)&=f_{n}+f_{n-1}\binom{t}{1}+\cdots + f_{r+1}\binom{t+n-r-2}{n-r-1}.
\end{align*}
In other words, for any $j\in [r+1,n]$, we have
\[
f_j(t) = \sum_{i=0}^{j-r-1} f_{j-i}\binom{t+i-1}{i}.
\]
From definition, $f_j(t)$ is exactly the associated $f$ polynomial to $u_0$ when considered as a monomial in $S_j$. In particular, this, and also by definition, we have
\begin{equation}\label{eq:f(t)}
    f(t)=f_n(t) = \sum_{i=0}^{n-r-1} f_{n-i}\binom{t+i-1}{i}.
\end{equation}
The next task is to determine $h(t)$ and $k(t)$. From now on, $t$ is assumed to be sufficiently large. As all the functions we are going to compute are polynomials (\cite[Corollary 4.13]{BE24}), it suffices to compute them for only sufficiently large values of $t$. For any $j\in [r+1,n]$, let $z_{j}(t)$ be the smallest monomial such that
\[
\pi_{j-1}\left( \mu_n(u_0x_n^t,z_{j}(t)) \right) = \pi_{j-1}(\mg(u_0x_n^t)).
\]
For sufficiently large $t$, we set
\begin{align*}
    h_j(t)&=\deg_{x_j} \mu_n(u_0x_n^t,z_j(t)),\\
    k_j(t)&=\deg_{x_j} z_j(t),
\end{align*}
for each $j\in [r+1,n]$. Observe that by definition, $h_j(t)$ and $k_j(t)$ are  exactly the associated $h(t)$ and $k(t)$ polynomials to $u_0$ when considered as a monomial in $S_j$, respectively.

Next, for each $j\in[r+1,n]$, set
\[
\delta_j= f_j(0)-h_j(0)=f_j(t)-h_j(t)
\]
for any $t$. This definition is well-defined due to \cite[Proposition 4.12]{BE24}. We now proceed to prove many formulae leading to $f(t),h(t),$ and $k(t)$.

\begin{lemma}\label{lem:smallest-cost-formula}
    Given positive integers $n,r,j,a,b,c$ where $r<n$ and $j\in [r+1,n]$. Let $z(t)\geq x_r^ax_n^{t-c}$ be the smallest monomial such that
    \[
    \pi_j\left( \mu_n\left(x_r^ax_n^{t-c}, z(t) \right)\right) = x_j^b.
    \]
    Then
    \begin{align*}
        z(t) &=  x_r^{a+b} x_{n}^{t-c-b},\\
        \mu_n\left(x_r^ax_n^t, z(t) \right) &=\left(x_j^b\right)\left( \prod_{s=j+1}^n x_s^{\binom{t-c+s-j}{s-j+1} -\binom{t-c-b+s-j}{s-j+1}}  \right).
    \end{align*}
\end{lemma}

\begin{proof}
    Without loss of generality, we can assume that $c=0$ as the result would follow from this specific case by replacing $t$ with $t-c$.
    
    We shall perform the steps needed to almost reach this cost:
    \begin{align*}
        u_0x_n^t=x_r^ax_n^t&\xrightarrow{x_n^t} x_r^a x_{n-1}^{t}\\
        x_r^ax_{n-1}^t&\xrightarrow{x_{n-1}^tx_n^{\binom{t}{2}}} x_r^a x_{n-2}^{t}\\
        x_r^ax_{n-2}^t&\xrightarrow{x_{n-2}^tx_{n-1}^{\binom{t}{2}}x_{n}^{\binom{t+1}{3}}} x_r^a x_{n-3}^{t}\\
        \cdots\\
        x_r^ax_{j+1}^t&\xrightarrow{x_{j+1}^tx_{j+2}^{\binom{t}{2}}\cdots x_{n}^{\binom{t+n-j-2}{n-j}  }} x_r^a x_{j}^{t}\\
        x_r^ax_{j}^t&\xrightarrow{x_{j}^{b-1}x_{j+1}^{\binom{t}{2}-\binom{t-b+1}{2}} \cdots x_{n}^{ \binom{t+n-j-1}{n-j+1} -\binom{t-b+n-j}{n-j+1} }  } x_r^{a+b-1} x_{j}^{t-b+1}.
    \end{align*}
    As can be seen from the above process, we have
    \[
    \pi_{j}\left(\mu_n\left(u_0x_n^t, x_r^{a+b-1} x_{j}^{t-b+1}\right) \right)= x_{j}^{b-1},
    \]
    and thus one more step is enough to obtain the desired cost. Therefore, we have
    \begin{align*}
        z(t)&= x_r^{d+b} x_{n}^{d-b},
    \end{align*}
    as desired. Next we determine $\deg_{x_s}\mu_n\left( u_0x_n^t  , z(t) \right)$ for each $s\in [n]$. It is immediate from the above process that $\deg_{x_{j}}\mu_n\left( u_0x_n^t  , z(t) \right)= b$ and $\deg_{x_{s}}\mu_n\left( u_0x_n^t  , z(t) \right)= 0$ for any $s\in [j-1]$. Next, for each $s\in [j+1,n]$, we have
    \begin{align*}
        \deg_{x_s}\mu_n\left( u_0x_n^t  , z(t) \right)&=t+ \sum_{u=2}^{s-j} \binom{t+u-2}{u} + \binom{t+s-j-1}{s-j+1} -\binom{t-b+s-j}{s-j+1} \\
        &=\sum_{u=0}^{s-j+1} \binom{t+u-2}{u} -\binom{t-b+s-j}{s-j+1}\\
        &= \binom{t+s-j}{s-j+1} -\binom{t-b+s-j}{s-j+1}, 
    \end{align*}
    as desired, where the last equality is due to \cref{lem:comb-2}.
\end{proof}

\begin{proposition}\label{prop:all-big-expressions}
    We have
    \begin{align*}
        z_{r+1}(t)&= u_0x_n^t,\\
        \mu_n(u_0x_n^t,z_{r+1}(t))&= 0,\\
        h_{r+1}(t)&=0,\\
        \delta_{r+1}&=f_{r+1}.
    \end{align*}
    Moreover, for any $j\in[r+2,n]$, we have
    \begin{align*}
        z_j(t) &= x_r^{d+\sum_{i=r+1}^{j-1}\delta_{i}}x_n^{t-\sum_{i=r+1}^{j-1}\delta_{i}},\\
        \mu_n(u_0x_n^t,z_j(t))&=\left(\prod_{i=r+1}^{j-1}x_{i}^{f_{i}(t)}\right)\left( \prod_{s=j}^n x_s^{g_{j,s}(t)}  \right),\\
        h_j(t) &= \begin{multlined}[t]
            \binom{t+j-r-1}{j-r} - \binom{t-\sum_{l=r+1}^{j-1}\delta_{l}+1}{2} -\\
            \sum_{i=3}^{j-r}\binom{t-\sum_{l=r+1}^{j+1-i}\delta_{l}+i-2}{i},
        \end{multlined}\\
        \delta_j &= f_j + \binom{-\sum_{l=r+1}^{j-1}\delta_{l}+1}{2} + \sum_{i=3}^{j-r}\binom{-\sum_{l=r+1}^{j+1-i}\delta_{l}+i-2}{i},
    \end{align*}
    where 
    \[
    g_{j,s}(t)=\begin{multlined}[t]
        \binom{t+s-r-1}{s-r} - \binom{t-\sum_{l=r+1}^{j-1}\delta_{l}+s-j+1}{s-j+2} -\\
        \sum_{i=3}^{j-r}\binom{t-\sum_{l=r+1}^{j+1-i}\delta_{l}+s-j+i-2}{s-j+i}
    \end{multlined}
    \]
    for each $s\in [j,n]$.
\end{proposition}

\begin{proof}
    We will prove the first statement directly and the second statement by induction on $j$. By definition, $z_{r+1}(t)\geq u_0x_n^t$ is the smallest monomial such that
    \[
    \pi_{r}\left(\mu_n\left( u_0x_n^t  , z_{r+1}(t) \right)\right)= \pi_r\left( \mg (u_0x_n^t) \right) = 0.
    \]
    Therefore, we have
    \begin{align*}
        z_{r+1}(t)&=u_0x_n^t,\\
        \mu_n\left( u_0x_n^t  , z_{r+1}(t) \right) &= 0,\\
        h_{r+1}(t)&= \deg_{x_{r+1}} \mu_n\left( u_0x_n^t  , z_{r+1}(t) \right) = 0,\\
        \delta_{r+1}&= f_{r+1}(0)-h_{r+1}(0)= f_{r+1}. 
    \end{align*}
    This proves the first statement.
    
    For the second statement, we first prove the base case $j=r+2$. By definition, $z_{r+2}(t)\geq u_0x_n^t$ is the smallest monomial such that \[\pi_{r+1}\left(\mu_n\left( u_0x_n^t  , z_{r+2}(t) \right)\right)= \pi_{r+1}\left( \mg (u_0x_n^t) \right)= x_{r+1}^{f_{r+1}(t)} = x_{r+1}^{f_{r+1}} = x_{r+1}^{\delta_{r+1}} .\] 
    Applying Lemma~\ref{lem:smallest-cost-formula}, we obtain
    \begin{align*}
        z_{r+2}(t)&=x_r^{d+\delta_{r+1}}x_n^{t-\delta_{r+1}},\\
        \mu_n\left( u_0x_n^t  , z_{r+2}(t) \right) &= \left( x_{r+1}^{\delta_{r+1}} \right)\left( \prod_{s=r+2}^n x_s^{ \binom{t+s-r-1}{s-r} - \binom{t-\delta_{r+1}+s-r-1}{s-r}  } \right) ,
    \end{align*}
    as desired. Then
    \begin{align*}
        h_{r+2}(t)&= \deg_{x_{r+2}} \mu_n\left( u_0x_n^t  , z_{r+2}(t) \right) = \binom{t+1}{2} -\binom{t-f_{r+1}+1}{2},\\
        \delta_{r+2}&= f_{r+2}(0)-h_{r+2}(0)= f_{r+2}+\binom{-\delta_{r+1}+1}{2},
    \end{align*}
    as desired. This concludes the proof of the base case of the induction. By induction, we can now assume that the formulas for $z_{j-1}(t), \mu_n(u_0x_n^t,z_{j-1}(t)), h_{j-1}(t),$ and $\delta_{j-1}$ are as predicted.

    By definition, $z_{j}(t)\geq u_0x_n^t$ is the smallest monomial such that 
    \[
    \pi_{j-1}\left(\mu_n\left( u_0x_n^t  , z_{j}(t) \right)\right)= \pi_{j-1}\left( \mg (u_0x_n^t) \right)= \prod_{i=r+1}^{j-1} x_{i}^{f_{i}(t)}.
    \] 
    Since 
    \[
    \pi_{j-1}\left(\mu_n\left( u_0x_n^t  , z_{j-1}(t) \right)\right) = \left( \prod_{i=r+1}^{j-2} x_i^{f_i(t)} \right)\left( x_{j-1}^{ h_{j-1}(t)} \right),\]
    the monomial $z_{j}(t)\geq z_{j-1}(t)$ is the smallest monomial such that
    \[
    \pi_{j-1}\left(\mu_n\left(  z_{j-1}(t) , z_{j}(t) \right)\right)= x_{j-1}^{f_{j-1}(t)-h_{j-1}(t)}=x_{j-1}^{f_{j-1}(0)-h_{j-1}(0)}=x_{j-1}^{\delta_{j-1}}.
    \]
    Applying Lemma~\ref{lem:smallest-cost-formula}, we obtain
    \begin{align*}
        z_j(t)&=x_r^{d+\sum_{i=r+1}^{j-1}\delta_i} x_n^{d-\sum_{i=r+1}^{j-1}\delta_i},\\
        \mu_n\left(  z_{j-1}(t) , z_{j}(t) \right)&=\left( x_{j-1}^{\delta_{j-1}} \right)\left( \prod_{s=j}^n x_s^{ \binom{t-\sum_{l=r+1}^{j-2}\delta_l+s-j+1}{s-j+2} -\binom{t-\sum_{l=r+1}^{j-1}\delta_l +s-j+1 }{s-j+2} } \right).
    \end{align*}
    The formula for $z_j(t)$ is now proved. Next we work on $\mu_n\left(  u_0x_n^t , z_{j}(t) \right)$. For any $s\in [n]$, we have
    \[
    \deg_{x_s} \mu_n\left(  u_0x_n^t , z_{j}(t) \right) = \deg_{x_s} \mu_n\left(  u_0x_n^t , z_{j-1}(t) \right) + \deg_{x_s} \mu_n\left(  z_{j-1}(t) , z_{j}(t) \right),
    \]
    and both of these are known from the above formula and the induction hypothesis. When $s\leq j-1$, the result follows from our process. Now assume that $s\in [j,n]$. We have
    \begin{align*}
        &\ \ \  \deg_{x_s} \mu_n\left(  u_0x_n^t , z_{j-1}(t) \right) \\
        &= \begin{multlined}[t]
            \binom{t+s-r-1}{s-r} - \binom{t-\sum_{l=r+1}^{j-2}\delta_l+s-j+2}{s-j+3}-\\
            \sum_{i=3}^{j-r-1}\binom{t-\sum_{l=r+1}^{j-i}\delta_l+s-j+i-1}{s-j+i+1} \
        \end{multlined}
    \end{align*}
    and 
    \begin{align*}
        &\ \ \ \deg_{x_s} \mu_n\left(  z_{j-1}(t) , z_{j}(t)\right)\\
        &=          \binom{t-\sum_{l=r+1}^{j-2}\delta_l+s-j+1}{s-j+2} -\binom{t-\sum_{l=r+1}^{j-1}\delta_l +s-j+1 }{s-j+2}
    \end{align*}
    We combine the second summand of $\deg_{x_s} \mu_n\left(  u_0x_n^t , z_{j-1}(t) \right)$ and the first summand of $\deg_{x_s} \mu_n\left(  z_{j-1}(t) , z_{j}(t)\right)$:
    \begin{align*}
         &\ \ \ - \binom{t-\sum_{l=r+1}^{j-2}\delta_l+s-j+2}{s-j+3}+ \binom{t-\sum_{l=r+1}^{j-2}\delta_l+s-j+1}{s-j+2} \\
         & = - \binom{t-\sum_{l=r+1}^{j-2}\delta_l+s-j+1}{s-j+3}.
    \end{align*}
    Therefore, we have
    \begin{align*}
        &\ \ \ \ \deg_{x_s} \mu_n\left(  u_0x_n^t , z_{j}(t) \right)\\
        &= \deg_{x_s} \mu_n\left(  u_0x_n^t , z_{j-1}(t) \right) + \deg_{x_s} \mu_n\left(  z_{j-1}(t) , z_{j}(t) \right)\\
        &=\begin{multlined}[t]
            \left(\binom{t+s-r-1}{s-r} - \sum_{i=3}^{j-r-1}\binom{t-\sum_{l=r+1}^{j-i}\delta_l+s-j+i-1}{s-j+i+1} \right) +\\
            \left(  -\binom{t-\sum_{l=r+1}^{j-1}\delta_l +s-j+1 }{s-j+2} \right) + \left( -\binom{t-\sum_{l=r+1}^{j-2}\delta_l+s-j+1}{s-j+3}  \right)
        \end{multlined}.
    \end{align*}
    We rewrite the second summand of the above expression:
    \begin{equation}\label{eq:big-expression}
        - \sum_{i=3}^{j-r-1}\binom{t-\sum_{l=r+1}^{j-i}\delta_l+s-j+i-1}{s-j+i+1} = - \sum_{i=4}^{j-r}\binom{t-\sum_{l=r+1}^{j-i+1}\delta_l+s-j+i-2}{s-j+i}.
    \end{equation}
    Now it is clear that 
    \[
    -\binom{t-\sum_{l=r+1}^{j-2}\delta_l+s-j+1}{s-j+3} 
    \]
    can be added to the sum (\cref{eq:big-expression}) as it corresponds to the case $i=3$. Therefore,
    \begin{align*}
        &\ \ \ \ \deg_{x_s} \mu_n\left(  u_0x_n^t , z_{j}(t) \right)\\
        &=\begin{multlined}[t]
            \binom{t+s-r-1}{s-r} - \binom{t-\sum_{l=r+1}^{j-1}\delta_l +s-j+1 }{s-j+2} - \\
            \sum_{i=3}^{j-r}\binom{t-\sum_{l=r+1}^{j-i+1}\delta_l+s-j+i-2}{s-j+i},
        \end{multlined} 
    \end{align*}
    as desired. The formulas for $h_j(t)$ and $\delta_j$ then follow trivially.
\end{proof}

 We obtain a deceptively simple formula for $\tau_n(x_r^d)$.

\begin{theorem}\label{thm:main-3-short-and-succint}
    Given integers $n,r$ where $2\leq r<n$. Then the Gotzmann threshold of the monomial $x_r^d$ in $R_n$ is
    \begin{align*}
        \tau_{n}(x_r^d)= \sum_{i=r+1}^{n} \delta_i.
    \end{align*}
\end{theorem}

\begin{proof}
    Fix $n\geq r+1$. We will show that
    \[
    \tau_m(x_r^d) = \sum_{j=r+1}^m\delta_j
    \]
    for each $m\in [r+1,n]$. For the base case $m=r+1$, we have
    \[
    \tau_m(x_r^d) = f_{r+1}(0)-h_{r+1}(0)-k_{r+1}(0) = f_{r+1}=\delta_{r+1},
    \]
    as desired, where the second equality follows from \cref{prop:all-big-expressions}. Now assume that $m>r+1$. We have
    \begin{align*}
        \tau_m(x_r^d) &=\left( f_m(0)-h_m(0)\right)-k_m(0) =\left(\delta_{m} \right) + \sum_{j=r+1}^{m-1} \delta_j
        =\sum_{j=r+1}^{m} \delta_j,
    \end{align*}
    as desired, where the second equality is due to our induction hypothesis.
\end{proof}

We are now ready to prove the main result.

\begin{proof}[Proof of Theorem~\ref{thm:main-4-recursive}]
    Fix an $n\geq r+1$. We remark that the formula of $\delta_{r+1}$ follows the same formula for $\delta_j$ for $j\in [r+2,n]$. Thus we will use this uniform formula. For each $j\in [r+1,n]$, we rewrite $\delta_j$ as follows:
    \begin{align*}
        \delta_j &= f_j + \binom{-\sum_{l=r+1}^{j-1}\delta_{l}+1}{2} + \sum_{i=3}^{j-r}\binom{-\sum_{l=r+1}^{j+1-i}\delta_{l}+i-2}{i}\\
        &= f_j + \binom{-\sum_{l=r+1}^{j-1}\delta_{l}+1}{2} + \sum_{i=r+1}^{j-2}\binom{-\sum_{l=r+1}^{i}\delta_{l}+j-1-i}{j+1-i}\\
        &= f_j + \binom{-\tau_{j-1}(x_r^d)+1}{2} + \sum_{i=r+1}^{j-2}\binom{-\tau_{i}(x_r^d)+j-1-i}{j+1-i}.
    \end{align*}
    We now have
    \begin{align*}
        \tau_n(x_r^d)&= \sum_{j=r+1}^n \delta_j\\
        &=\sum_{j=r+1}^n \left(f_j + \binom{-\tau_{j-1}(x_r^d)+1}{2} + \sum_{i=r+1}^{j-2}\binom{-\tau_{i}(x_r^d)+j-1-i}{j+1-i}\right)\\
        &=\left(\sum_{j=r+1}^n f_j\right) + \sum_{j=r+1}^n  \left(  \binom{-\tau_{j-1}(x_r^d)+1}{2} + \sum_{i=r+1}^{j-2}\binom{-\tau_{i}(x_r^d)+j-1-i}{j+1-i}\right).
    \end{align*}
    Let $\Delta_s$ be the sum of the binomial coefficients whose upper indices contain $\tau_{s}(x_r^d)$, for each $s\in [r+1,n-1]$. With this definition, we have
    \[
    \tau_n(x_r^d) = \left(\sum_{j=r+1}^n f_j\right) + \sum_{s=r+1}^{n-1} \Delta_s.
    \]
    We will simplify each summand above. For the first expression, we have
    \begin{align*}
        &\ \ \ \ \sum_{j=r+1}^n f_j\\ &=\sum_{j=r+1}^n \left( \binom{j+d-2}{j-1} - \binom{j-r+d-1}{j-r}\right)\\
        &= \left( \sum_{j=r+1}^n\binom{j+d-2}{j-1} \right) - \left(\sum_{j=r+1}^n\binom{j-r+d-1}{j-r}\right).
    \end{align*}
    By adding $\binom{r+d-1}{r-1}$ to the first sum and $1=\binom{d-1}{0}$ to the second sum, we obtain
    \begin{align*}
        &\ \ \ \ \sum_{j=r+1}^n f_j\\ 
        &= \begin{multlined}[t]
            \left( \binom{r+d-1}{r-1}+\sum_{j=r+1}^n\binom{j+d-2}{j-1} \right) - \left( \binom{d}{0} +\sum_{j=r+1}^n\binom{j-r+d-1}{j-r}\right)-\\
            \left(\binom{r+d-1}{r-1}-1 \right)
        \end{multlined}\\
        &=\left( \binom{r+d-1}{r-1}+ \sum_{j=r+1}^n\binom{j+d-2}{j-1} \right) - \left( \sum_{j=r}^n\binom{j-r+d-1}{j-r}\right)-
            \left(\binom{r+d-1}{r-1}-1 \right)\\
        &=\left( \binom{r+d-1}{r-1}+ \sum_{j=r}^{n-1}\binom{j+d-1}{j} \right) - \left( \sum_{j=0}^{n-r}\binom{j+d-1}{j}\right)-
            \left(\binom{r+d-1}{r-1}-1 \right)
    \end{align*}
    We apply \cref{lem:comb-2} to the very first summand $\binom{r+d-1}{r-1}$ to obtain
    \begin{align*}
        &\ \ \ \ \sum_{j=r+1}^n f_j\\ 
        &=\left( \sum_{j=0}^{r-1} \binom{j+d-1}{j}+ \sum_{j=r}^{n-1}\binom{j+d-1}{j} \right) - \left( \sum_{j=0}^{n-r}\binom{j+d-1}{j}\right)-
            \left(\binom{r+d-1}{r-1}-1 \right)\\
        &=\left( \sum_{j=0}^{n-1} \binom{j+d-1}{j} \right) - \left( \sum_{j=0}^{n-r}\binom{j+d-1}{j}\right)-
            \left(\binom{r+d-1}{r-1}-1 \right)
    \end{align*}
    Applying \cref{lem:comb-2} to the first two summands, we obtain
    \begin{align*}
        \sum_{j=r+1}^n f_j&= \left( \binom{n+d-1}{n-1} \right) - \left( \binom{n+d-r}{n-r} \right)- \left(\binom{r+d-1}{r-1}-1 \right)\\
        &=\binom{d+n-1}{d} - \binom{d+n-r}{d} - \binom{d+r-1}{d} +1.
    \end{align*}
    Next, for each $s\in [r+1,n-2]$, we have
    \begin{align*}
        \Delta_s &=  \binom{-\tau_s(x_r^d)+1}{2} + \sum_{j=s+2}^n  \binom{-\tau_s(x_r^d)+j-1-s}{j+1-s} \\
        &= \binom{-\tau_s(x_r^d)+1}{2} + \sum_{j=3}^{n+1-s}  \binom{-\tau_s(x_r^d)+j-2}{j}.
    \end{align*}
    Applying \cref{lem:comb-2} to the first summand, we obtain
    \[
    \binom{-\tau_s(x_r^d)+1}{2} = \sum_{j=0}^2 \binom{-\tau_s(x_r^d)+j-2}{j}. 
    \]
    Thus we can combine the sum and apply \cref{lem:comb-2} again:
    \begin{align*}
        \Delta_s  &= \sum_{j=0}^{n+1-s}  \binom{-\tau_s(x_r^d)+j-2}{j}= \binom{-\tau_s(x_r^d)+n+1-s}{n+1-s} = (-1)^{n+1-s} \binom{\tau_s(x_r^d)}{n+1-s},
    \end{align*}
    where the last equality is due to the formula for negative binomial coefficients $\binom{-a}{b}=(-1)^b\binom{a+b-1}{b}$ where $a\geq 1$ and $b\geq 0$.
    Finally, we compute $\Delta_{n-1}$:
    \[
    \Delta_{n-1} = \binom{-\tau_{n-1}(x_r^d)+1}{2} = (-1)^{n+1-(n-1)}  \binom{\tau_{n-1}(x_r^d)}{n+1-(n-1)}.
    \]
    The result then follows.
\end{proof}

\bibliographystyle{amsplain}
\bibliography{refs}

\end{document}